\documentclass[12 pt]{article} %Russian ubrat'!
\usepackage[T1]{fontenc} 
\usepackage{ulem} 
\usepackage[utf8]{luainputenc} 
\usepackage{geometry} 
\usepackage[pdftex]{graphicx} 
\usepackage{amstext} 
\usepackage{amssymb} 
\usepackage{amsmath} 
\usepackage{physics}
\usepackage[T1,T2A]{fontenc} 
\usepackage[utf8]{inputenc} 
\usepackage{setspace} 
\usepackage{amsthm}
\usepackage{esint} 
\usepackage{comment} 
\usepackage{float} 
\usepackage{fullpage} 
\usepackage{tikz} 
\usepackage{indentfirst} 
\usepackage{mathtext}
\usepackage{ dsfont }
\usepackage{ amsfonts }
\usepackage{ eufrak }
\usepackage{ mathrsfs }
\usepackage{color}
 \usepackage{hyperref}

\newtheorem{Th}{Theorem}
\newtheorem{Co}{Corollary}
\newtheorem{Pp}{Proposition}
\newtheorem{Le}{Lemma}

\newif\ifeng
\engfalse

\newif\ifrus
\rustrue % \rustrue or \rusfalse
\newif\ifhide
\hidefalse

\ifrus

\title{{\normalsize\tt\hfill\jobname.tex}\\
On recurrence and convergence rate 
of generalised Fisher -- Wright diffusion with mutation}
\author{R.Yu. Sineokiy\footnote{M.V. Lomonosov Moscow State University, Moscow, Russian Federation; email: ro-ssi@mail.ru} \; \& \, A.Yu. Veretennikov\footnote{Institute for Information Transmission Problems (Kharkevich Institute) RAS, Moscow, Russian Federation; email: ayv@iitp.ru} 
} 
\begin{document} 

\fi

\ifeng
\title{\normalsize\tt\hfill\jobname.tex}\\
On recurrent properties and convergence rates of generalised Fisher -- Wright's diffusion with mutation}
\author{A.Yu. Veretennikov\footnote{IITP RAS} \; \& \,
R.Yu. Sineokiy \footnote{MSU}} 
%\raggedbottom 
%%\begin{document} 
%\thispagestyle{empty} 

\fi

\maketitle

%\tableofcontents

\begin{abstract}
A generalised one-dimensional Wright -- Fisher diffusion process with mutations is considered which describes a popular population genetics model. An exponential recurrence and exponential convergence towards a unique invariant measure are established under minimal regularity conditions of the coefficients. 
%A bound for the beta-mixing coefficient is also obtained.

\medskip


\medskip


\end{abstract}

\section{Introduction}

A one-dimensional diffusion process is considered which generalises a well-known Wright-Fisher diffusion with mutations, studied earlier in \cite{WFM1}, \cite{WFM2}, et al. The goal of this paper is to establish exponential ergodicity for the model under minimal requirements on the coefficients. Both ergodicity and convergence rate are based on two features: on a finite exponential moment of the first hitting time of some compact in $(0,1)$ and on local mixing. Lyapunov functions are used to estimate moments (which are not exponential ones). Local mixing is usually derived by using coupling in multi-dimensional cases,but in dimension one it suffices to use an easier tool, intersection times for two independent diffusion trajectories, see \cite[Section 2.1]{AYV21}, which simplifies the presentation considerably.

Results in \cite{WFM2} are devoted to the calculation of functionals for various versions of the Wright -- Fisher diffusion via analytical methods. No exponential convergence rates for the case (\ref{AB}) -- (\ref{Asigma2}) have been studied there, although, for a -- quite different from ours --  special ``neutral'' case certain asymptotic exponential formulae have been obtained, see \cite[section 3]{WFM2}. At least in some part of the paper \cite{WFM2} the coefficients are assumed differentiable, even though for studying recurrence such an assumption could have been possibly dropped; however, recurrence was not the topic in the cited paper. 

One more analytic study of the one-dimensional Wright -- Fisher diffusion from the point of view of its infinitesimal generator is in the paper  \cite{EpsteinMazzeo}. See, in particular, Corollary 6 where an exponential stabilisation of the solution of the corresponding heat equation is established. The assumptions of this paper do include certain regularity of the coefficients. Although, it is likely that some of them may be relaxed, yet the usage of the parametrix technique may hardly allow to get rid of them completely. Also, the methods applied in this paper as well as of the paper \cite{WFM2} are totally different from ours. 

The main results in what follows are theorems \ref{T1} and \ref{T2}, and also corollaries \ref{Cor1} and \ref{Cor2}. A crucial role in their proofs belongs to propositions \ref{Pro1} and \ref{Pro2}. In turn, lemma  \ref{L1} is important for their proofs. All of them together with proofs are in the next section. 
%The last section \ref{sec3} contains some bound for the beta-mixing coefficient. It is rather close to the bound for the rate of convergence in section \ref{sec2}, and yet, requires some additional work; this bound (as well as some other mixing inequalities) may be useful in certain limit theorems, as noted, for example, in \cite[Chapter ?]{IbrLin}.

\section{Main results}\label{sec2}
Consider a one-dimensional SDE 
\begin{equation}\label{eq1}
dX_t=B(X_t)dt+\epsilon\sigma(X_t)dW_t, \quad t\ge 0, \qquad X_0=x\in(0,1), 
\end{equation}
where $B(\cdot), \sigma(\cdot)$ are bounded Borel - measurable functions on $[0,1]$, $\epsilon>0$, the initial value  $x\in(0,1)$ is non-random, $W_t$ is a standard one-dimensional Wiener process. Let us assume that there exist constant values $\beta_0>0, b_0>0, \beta_1>0, b_1<0$ such that 
\begin{equation}\label{AB}
B(x)\geq b_0, x\in[0,\beta_0), \quad 
B(x)\leq b_1, x\in(1-\beta_1,1],
\end{equation}
and 
\begin{equation}\label{Asigma}
0<\sigma(x)^2 \leq \mu x(1-x), x \in (0,1), 
\end{equation}
and in addition, 
\begin{equation}\label{Asigma2}
\inf_{x_0\le x\le 1-x_0}\sigma(x)^2 >0
\end{equation}
for any  $0<x_0<1/2$.

The existence of weak solution follows from the results in  \cite{Krylov1969}, including in the multi-dimensional case with a non-degenerate diffusion coefficient. More than that, in the one-dimensional case it follows from 
\cite{Krylov1969} via localization and a suitable limit it may be concluded that this solution is weakly unique (that is, unique in distribution) and, hence, is a Markov and strong Markov process. Further, conditions for non-attainability of the end-points of $[0,1]$ by the process $X_t$ will be imposed. According to (\ref{Asigma2}) the diffusion coefficient is uniformly non-degenerate on any compact in $(0,1)$; then it will follow from the next lemma \ref{L1} that $X_t\not=0$ for each  $t\geq 0$; at the same time it is not true that it remains bounded away from zero if $|x-1/2|\uparrow 1/2$. 

Further, the proof of Feller's result is provided for completeness of this presentation about a non-attainability of the end-points \cite{Feller2} for the model under consideration on the basis of the idea from \cite{Feller2} and one hint from the preprint \cite{IIG} in which a slightly different CIR model was studied. Note that in \cite{Feller2} the drift coefficient was ``for simplicity'' assumed continuous, while the diffusion coefficient belonged to  $C^1$, which is not the case in the present paper. Yet, of course, it is essentially Feller's result.

\begin{Le}\label{L1}
Let conditions  (\ref{AB})--(\ref{Asigma2}) be satisfied and 
\begin{equation}\label{AF0}
\frac{2b_0}{\epsilon^2\mu}-1>0,
\end{equation}
and also 
\begin{equation}\label{AF1}
\frac{2b_1}{\epsilon^2\mu}+1<0.
\end{equation}
Then the process $X_t$ does not equal 0 or 1 for any $t\ge 0$. \end{Le}
\begin{proof}
Let us consider the end-point zero. Let $m=\frac{2b_0}{\epsilon^2\mu}-1$ and consider the function 
$V(x)=x^{-m}$. By Ito's formula, 
\[
dX_t^{-m}=[-mB(X_t)X_t^{-m-1}+\frac{m(m+1)\epsilon^2}{2}\sigma(X_t)^2X_t^{-m-2}]dt - m\epsilon \sigma(X_t)X_t^{-m-1}dW_t. 
\]
For any $0\le \alpha<\beta_0$ let
\begin{align*}
&\tau_\alpha=\inf\{t\geq 0: X_t\leq \alpha\}, \;
T_{\beta_0}=\inf\{t\geq 0: X_t \geq \beta_0\}, 
 \\\\
&\tau_{\alpha, \beta_0}=\min(\tau_\alpha,T_{\beta_0}), \; \tau_{t,\alpha, \beta_0}=\min(\tau_{\alpha, \beta_0}, t).
\end{align*}
We have for any $0<\alpha < \beta_0$ small enough, 
\begin{align*}
&E_xX^{-m}_{\tau_{t,\alpha, \beta_0}}=x^{-m} + E_x\int_{0}^{\tau_{t,\alpha, \beta_0}}{[-mB(X_s)X_s^{-m-1}+\frac{m(m+1)\epsilon^2}{2}\sigma(X_s)^2X_s^{-m-2}]}ds
 \\\\
&\leq x^{-m} + E_x\int_{0}^{\tau_{t,\alpha, \beta_0}}{[-mb_0X_s^{-m-1}+\frac{m(m+1)\epsilon^2\mu}{2}X_s(1-X_s)X_s^{-m-2}]}ds
 \\\\
&=x^{-m} + E_x\int_{0}^{\tau_{t,\alpha, \beta_0}}{[(-mb_0+\frac{m(m+1)\epsilon^2\mu}{2})X_s^{-m-1}-\frac{m(m+1)\epsilon^2\mu}{2}X_s^{-m}]}ds
 \\\\
&=x^{-m} + E_x\int_{0}^{\tau_{t,\alpha, \beta_0}} \{-\frac{m(m+1)\epsilon^2\mu}{2}X_{\tau_{s,N}}^{-m}\} ds \leq x^{-m}.
\end{align*}
It was used that $X_s^{-m-1} \ge \alpha^{-1} X_s^{-m}$ for any $0\le s \le \tau_{t,\alpha, \beta_0}$, while $-mb_0+\frac{m(m+1)\epsilon^2\mu}{2} <0$ due to the assumptions of the lemma and by virtue of the choice of the value  $m$. 
 
It is well-known that $\tau_{\alpha, \beta_0}<\infty$ a.s.. So, by Fatou's lemma we get as $t\to\infty$ that 
$$
E_xX^{-m}_{\tau_{\alpha, \beta_0}} \le x^{-m}.
$$
Due to the equality 
$$
E_xX^{-m}_{\tau_{\alpha, \beta_0}} 
= \alpha^{-m} P_x (X_{\tau_{\alpha, \beta_0}} = \alpha) 
+ \beta_0^{-m} P_x (X_{\tau_{\alpha, \beta_0}} = \beta_0), 
$$
it may be now concluded that 
$$
P_x (X_{\tau_{\alpha, \beta_0}} = \alpha)  \to 0, \quad \alpha\to 0.
$$
Let us assume that $P(\tau_0 <\infty)>0$. Then on a set $\Omega_0:= \{\omega: \tau_0 <\infty\}$ of a non-zero probability there is a convergence 
$$
\tau_{\alpha, \beta_0} \to \tau_{0, \beta_0}, \quad \alpha \to 0. 
$$
Therefore, since the trajectories of the process are continuous, we have on this set that 
$$
X_{\tau_{\alpha, \beta_0}} \to X_{\tau_{0, \beta_0}}, \quad \alpha\to 0.
$$ 
However, for each $\alpha>0$
$$
P_x (X_{\tau_{0, \beta_0}} = 0) \le P_x (X_{\tau_{\alpha, \beta_0}} = \alpha).
$$
Since the right hand side here tends to zero as  $\alpha\to 0$, while the left hand side does not depend on $\alpha$, we conclude that $$
P_x (X_{\tau_{0, \beta_0}} = 0) = 0. 
$$
From here it follows straightforwardly that 
$$
P_x (\tau_{0, \beta_0} <\infty) = 0. 
$$
Indeed, if $X_0= x\in (0,\beta_0)$, then the process hits the boundary of the interval $x\in (0,\beta_0)$ almost surely at $\beta_0$. Now, to attain ever the origin, it must again take the value $x$ at some time. But as long as it equals $x$, the process will again exit from the interval $(0,\beta_0)$ at $\beta_0$ with probability one. Thus, with probability one it never equals zero. The case of the boundary point 1 follows from similar considerations. The lemma is proved. 
\end{proof}

The combination of proposition \ref{Pro1} and its complementary one and kind of ``symmetric'' proposition \ref{Pro2} is, in fact, the basis for the proof of the main result , theorem \ref{T1}, which will be stated in what follows. 
Let 
$$
T_\alpha:= \inf(t\ge 0: X_t \ge \alpha), 
\; T'_\alpha:= \inf(t\ge 0: X_t \le 1-\alpha), \;
\hat T_\alpha = T_\alpha \wedge T'_\alpha.
$$

\begin{Pp}\label{Pro1}
Let conditions  (\ref{AB})--(\ref{AF1}) be satisfied. then for any $c>0$ there exist $0<\alpha<\beta_0$, $m>0$ such that for any $x<\alpha$ the following inequality holds true, 
\begin{equation}\label{Pro1-1}
E_xe^{c T_\alpha}\leq C(m)c\alpha^{m+1} x^{-m}+1,
\end{equation}
where
$$
C(m)=\frac{2}{mb_0-\frac{\epsilon^2}{2}m(m+1)\mu}.
$$
More than that, 
\begin{equation}\label{P1-2}
E_x\int_{0}^{ T_\alpha} X_s^{-m-1}ds\leq C(m) x^{-m}.
\end{equation}
Vice versa, for any $\alpha>0$ small enough there exists $c>0$ such that inequality (\ref{Pro1-1}) is valid. 
\end{Pp}

\begin{proof}
%Рассмотрим функцию $V(x,t)=e^{ct}x^{-m}$. 
Let $\alpha < \beta_0$. We have by Ito's formula, 
\[
de^{ct}X_t^{-m}= e^{ct}[cX_t^{-m}-mB(X_t)X_t^{-m-1}+\frac{m(m+1)\epsilon^2}{2}\sigma(X_t)^2X_t^{-m-2}]dt - m\epsilon \sigma(X_t)e^{ct}X_t^{-m-1}dW_t.
\]
Let
$$
T_{N,\alpha}:=\inf\{t\ge 0: X_t \leq \frac{1}{N}, \; \text{or}\; X_t\ge \alpha\}, \quad T_{t,N,\alpha}:= T_{N,\alpha} \wedge t, 
$$ 
with $N>1$. Then we get
\begin{equation}\label{eP1-1}
E_xe^{c T_{t,N,\alpha}}X_{T_{t,N,\alpha}}^{-m}-x^{-m}=E_x\int_{0}^{T_{t,N,\alpha}} {e^{cs}[cX_s^{-m}-mB(X_s)X_s^{-m-1} + \frac{m(m+1)\epsilon^2}{2}\sigma(X_s)^2X_s^{-m-2}]ds}.
\end{equation}
Since $Ee^{cT_{t,N,\alpha}}X_{T_{t,N,\alpha}}^{-m} \geq 0$ then the following inequality is true, 
\begin{equation}\label{xge}
x^{-m}\geq E_x\int_{0}^{T_{t,N,\alpha}} e^{cs}[-cX_s^{-m}+mB(X_s)X_s^{-m-1} - \frac{m(m+1)\epsilon^2}{2}\sigma(X_s)^2X_s^{-m-2}]ds.
\end{equation}
Here the second term under the integral is non-negative for $\alpha<\beta_0$. Hence, choosing $\alpha<\beta_0$ and using the condition on $\sigma(x)^2$ we get, 
\begin{align*}
&E_x\int_{0}^{T_{t,N,\alpha}} {e^{cs}[-cX_s^{-m}+mB(X_s)X_s^{-m-1} - \frac{m(m+1)\epsilon^2}{2}\sigma(X_s)^2X_s^{-m-2}]ds} 
 \\\\
&\geq E_x\int_{0}^{T_{t,N,\alpha}} {e^{cs}[-cX_s^{-m}+mb_0X_s^{-m-1} - \frac{m(m+1)\epsilon^2}{2}\mu X_s(1-X_s)X_s^{-m-2}]ds}
 \\\\
&=E_x\int_{0}^{T_{t,N,\alpha}} {e^{cs}[(mb_0 - \frac{m(m+1)\epsilon^2\mu}{2})X_s^{-m-1}-(c-\frac{m(m+1)\epsilon^2\mu}{2})X_s^{-m}]ds}.
\end{align*}
Let us now choose $\alpha$ so that, by virtue of the bound $X_s\leq\alpha$ for any $s\in[0,T_{t,N,\alpha}]$, we obtain 
\[
(mb_0 - \frac{m(m+1)\epsilon^2\mu}{2}) \geq 2(c-\frac{m(m+1)\epsilon^2\mu}{2})X_s.
\]
This may be achieved if the following two conditions hold true simultaneously:
\begin{equation}
\alpha<\min\left(\beta_0,\frac{(mb_0 - \frac{m(m+1)\epsilon^2\mu}{2})}{2(c-\frac{m(m+1)\epsilon^2\mu}{2})}\right),
\label{eq:ref13}
\end{equation}
and 
\begin{equation}
 \left(mb_0 - \frac{m(m+1)\epsilon^2\mu}{2}\right)>0.
\label{eq:ref15}
\end{equation}
Condition (\ref{eq:ref15}) means that it suffices to choose any value $m$ from the interval 
\begin{equation}\label{m}
m\in \left(0,\frac{2b_0}{\epsilon^2\mu}-1\right).
\end{equation}
For such a choice, we get
\begin{align}\label{eNew-1}
&E_x\int_{0}^{T_{t,N,\alpha}}{e^{cs}[(mb_0 - \frac{m(m+1)\epsilon^2\mu}{2})X_s^{-m-1}-(c-\frac{m(m+1)\epsilon^2\mu}{2})X_s^{-m}]ds} \nonumber
 \\\nonumber\\\nonumber
&\geq E_x\int_{0}^{T_{t,N,\alpha}}{e^{cs}[(mb_0 - \frac{m(m+1)\epsilon^2\mu}{2})X_s^{-m-1} - \frac{1}{2}(mb_0 - \frac{m(m+1)\epsilon^2\mu}{2})X_s^{-m-1}]ds}
 \\\nonumber\\
&=E_x\int_{0}^{T_{t,N,\alpha}}{e^{cs}[\frac{1}{2}(mb_0 - \frac{m(m+1)\epsilon^2\mu}{2})X_s^{-m-1}]ds} \ge 0.
\end{align}
Here the expression in the square brackets in the last equality is positive because of the condition (\ref{eq:ref15}). So, we get by virtue of (\ref{xge}) the following:
\[
E_x\int_{0}^{T_{t,N,\alpha}}{e^{cs}[\frac{1}{2}(mb_0 - \frac{m(m+1)\epsilon^2\mu}{2})X_s^{-m-1}]ds}\leq x^{-m}.
\]
Since $X_s\leq\alpha$ and $e^{cs} \geq 1$ for any $s\in[0,T_{t,N,\alpha}]$, we obtain 
\begin{equation}
E_x\int_{0}^{T_{t,N,\alpha}}{e^{cs}ds}\leq \frac{2}{mb_0-\frac{\epsilon^2}{2}m(m+1)\mu} \alpha^{m+1} x^{-m}\label{eq:ref16}
\end{equation}
and
\begin{equation}
E_x\int_{0}^{T_{t,N,\alpha}} X_s^{-m-1}ds\leq \frac{2}{mb_0-\frac{\epsilon^2}{2}m(m+1)\mu} x^{-m}.
\label{eq:ref17}\end{equation}
Let us denote the fraction in the right hand sides of (\ref{eq:ref16}) and  (\ref{eq:ref17}) by  $C(m)$; then, it follows by integration of the left hand side of (\ref{eq:ref16}) that 
\[
E_xe^{cT_{t,N,\alpha}}\leq C(m)c\alpha^{m+1} x^{-m}+1.
\]
Therefore, by Fatou's lemma and because $T_{t,N,\alpha}\rightarrow T_{\alpha}$ as $t\rightarrow\infty$, $N \rightarrow \infty$, we conclude that 
\[
E_xe^{cT_{\alpha}}\leq C(m)c\alpha^{m+1} x^{-m}+1.
\]
Similarly, applying Fatou's lemma to the inequality (\ref{eq:ref17}), we obtain for any $x<\alpha$, 
\begin{equation}\label{eq:ref6}
E_x\int_{0}^{T_\alpha} X_s^{-m-1}ds\leq C(m) x^{-m},
\end{equation}
as required. 

Notice that the same calculus justifies that for any $\alpha>0$ small enough it is possible to choose $c>0$ such that the same bound (\ref{Pro1-1}) holds true. Proposition \ref{Pro1} is proved. 
\end{proof}

\begin{Pp}\label{Pro2}
Let conditions (\ref{AB})--(\ref{AF1}) be satisfied. Then for any $c>0$ there exist $0<\alpha<\frac{1}{2}$, $m>0$ such that  
\begin{equation}\label{P2-1}
E_xe^{c T'_\alpha}\leq C(m)c\alpha^{m+1} (1-x)^{-m}+1,
\end{equation}
for any $x>1-\alpha$, where
$$
C(m)=\frac{2}{-mb_1-\frac{\epsilon^2}{2}m(m+1)\mu}.
$$
Moreover, 
\begin{equation}\label{eq:ref12}
E_x\int_{0}^{T'_\alpha} (1-X_s)^{-m-1}ds\leq C(m)(1-x)^{-m}.
\end{equation}
More than that, for any $\alpha>0$ small enough there exists $c>0$ such that the inequality (\ref{P2-1}) holds true. 
\end{Pp}

\begin{proof} This proof is similar to the proof of the previous proposition. The easiest way to show it is to use the change of variables $x' = 1-x$ after which all assertions follow straightforwardly from proposition  \ref{Pro1} since all conditions on the process are symmetric with respect to such change of variables. The value $m$ here should be chosen from the interval 
\begin{equation}\label{m'}
m\in (0,\frac{-2b_1}{\epsilon^2\mu}-1).
\end{equation}
Proposition \ref{Pro2} follows. 
\end{proof}

Now, both propositions \ref{Pro1} and \ref{Pro2} allows to state the following first main result. 

\begin{Th}\label{T1}
Under the assumptions (\ref{AB})--(\ref{Asigma2})  for any $c>0$ there exist $0<\alpha<\frac{1}{2}$, $m>0$ such that 
\begin{equation}\label{exptau3}
E_xe^{c\hat T_\alpha}\leq C(m)c\alpha^{m+1}((1-x)^{-m}+x^{-m})+1,
\end{equation}
where 
\begin{equation}\label{Cm3}
C(m)=\max\left(\frac{2}{mb_0-\frac{\epsilon^2}{2}m(m+1)\mu},\frac{2}{-mb_1-\frac{\epsilon^2}{2}m(m+1)\mu}\right).
\end{equation}
More than that, 
\begin{equation}\label{T1-3}
E_x\int_{0}^{\hat T_\alpha} \left(X_s^{-m-1}+(1-X_s)^{-m-1}\right)ds\leq C(m)\left(x^{-m} + (1-x)^{-m}\right).
\end{equation} 
\end{Th}
\begin{proof}
To show the inequality (\ref{exptau3}) it suffices to take 
$\alpha = \min(\alpha_1,\alpha_2)$, $m = \min(m_1,m_2)$, where $\alpha_1,\alpha_2$, $m_1, m_2$  are the constants from propositions \ref{Pro1} and \ref{Pro2}, respectively. The bound (\ref{T1-3}) for  $x\not\in [\alpha,1-\alpha]$ follows from the combination of the inequalities  (\ref{eq:ref6}) and (\ref{eq:ref12}); for $x \in [\alpha,1-\alpha]$ it is trivial since in this case $\hat T_\alpha=0$. The constant $m$ should be chosen from the intersection of two intervals (\ref{m}) and (\ref{m'}), that is, 
\begin{equation}\label{m2}
m\in \left(0, \frac{(-2b_1)\wedge 2b_0}{\epsilon^2\mu}-1\right).
\end{equation}
The theorem follows. 
\end{proof}
Let us highlight that no regularity assumptions on the coefficients of the equation were assumed.

\begin{Co}\label{Cor1}
Under the assumptions of the theorem \ref{T1}, the Markov process satisfying equation (\ref{eq1}) with various (including distributed) initial conditions possesses at least one invariant measure $\mu$. 
\end{Co}
The uniqueness of this measure under the same assumptions also holds: this will be a part of the next theorem. 

\begin{proof}
The arguments are based on Khasminskii's method \cite[глава  4]{Khas}. 

{\bf 1.} Let 
\begin{equation}\label{a12}
0<\alpha_1<\alpha_2 < 1/2,
\end{equation} 
and let $\alpha_2$ be small enough. Let the initial data be  $x=\alpha_1$. Consider the stopping times  
$$
T_1 : = \inf(t\ge 0: \, X_t = \alpha_2, \, \text{либо}\, X_t = 1-\alpha_2), \quad 
T_2 : = \inf\left(t\ge T_1: \, \left\vert X_t - \frac12\right\vert = \frac12 - \alpha_1\right),
$$
(that is, $T_2 = \inf\left(t\ge T_1: \, X_t = \alpha_1, \; \text{or}\; X_t = 1 - \alpha_1\right)$), and further by induction as $n\ge 1$
$$
T_{2n+1} \!: =\! \inf(t\ge T_{2n}\!: \, X_t \!=\! \alpha_2, \, \text{либо}\, X_t \!=\! 1\!- \alpha_2), \; 
T_{2n+2} \!: =\! \inf\left(\!t\ge T_{2n+1}\!: \, \left\vert X_t \!- \!\frac12\right\vert \!=\!\frac12 \!- \alpha_1\!\right). 
$$
By virtue of theorem \ref{T1}, the random variable possesses some exponential moment,  $E\exp(cT_1)<\infty$. Due to the non-degeneracy of the diffusion coefficient on the interval $[\alpha_1, 1-\alpha_1]$, and because of a positive probability of exit from the interval  $(\alpha_1, 1-\alpha_1)$ over the unit of time, and since the process is strongly Markov, the random variable $T_2-T_1$ also possesses some exponential moment, say, $E\exp(c(T_2-T_1))<\infty$ with a certain  $c>0$. Therefore, by induction, in any case, the expectations and all moments of all random variables $T_n$ are finite. 

Now let us consider the homogeneous Markov chain $Y_n:= X_{T_{2n}}$ with two states, $S = \{\alpha_1, 1-\alpha_1\}$, and {\it with strictly positive} transition probabilities
$$
P(X_{T_{2n+2}} = \alpha_1\vert X_{T_{2n}}=\alpha_1), 
\quad P(X_{T_{2n+2}} = 1-\alpha_1\vert X_{T_{2n}}=\alpha_1), 
$$ 
and
$$
P(X_{T_{2n+2}} = \alpha_1\vert X_{T_{2n}}=1-\alpha_1), 
\quad P(X_{T_{2n+2}} = 1-\alpha_1\vert X_{T_{2n}}=1-\alpha_1). 
$$
As it is well-known, such a Markov chain has a (unique) invariant measure, say, $\nu$. 
 
\medskip

{\bf 2.} Now for any Borel measurable set $A\subset [0,1]$ let 
\begin{equation}\label{invmu}
\mu(A):= c_0 E_\nu \int_0^{T_2}1(X_t \in A)\, dt.
\end{equation}
This is exactly the construction from \cite[chapter 4]{Khas} which provides a positive invariant measure {\it for the process $X_t$,} where $c_0$ is a normalized constant. The fact that this measure is finite and, hence, may be normalized, follows from the finiteness of the expectations of the random variables $T_1$ and $T_2-T_1$. The invariance property follows verbatim from the calculus in the proof of theorem 4.4.1 \cite{Khas}; hence, we do not copy it here. Corollary \ref{Cor1} is proved.
\end{proof}

\begin{Co}\label{Cor2}
Under the assumptions of theorem \ref{T1} the following bound is valid for the invariant measure $\mu$ constructed earlier in corollary \ref{Cor1}:
\begin{equation}\label{muint}
I(m):=\int_0^1 \left(x ^{-m} + (1-x)^{-m}\right) \mu (dx) < \infty, 
\end{equation}
where $m$ is the constant chosen according to the condition (\ref{m2}).
\end{Co}
\begin{proof}
The inequality in question follows straightforwardly from the definition of $\mu$ in (\ref{invmu}) and from inequalities (\ref{eq:ref6}) and (\ref{eq:ref12}). 

\medskip

Indeed, we estimate for $g(x) = x ^{-m} + (1-x)^{-m}, \, x\in (0,1)$ due to the definition (\ref{invmu}) and by virtue of the bound (\ref{T1-3}), which is a result of the combination of (\ref{eq:ref6}) and (\ref{eq:ref12}),
\begin{align*}
&\langle g, \mu\rangle = c_0 E_\nu \int_0^{T_2}g(X_t)\, dt 
= c_0 E_\nu \int_0^{T_1}g(X_t)\, dt +  c_0 E_\nu \int^{T_2}_{T_1}g(X_t)\, dt
 \\\\
&\leq c_0\times C(m)\left(\alpha_2^{-m} + (1-\alpha_2)^{-m}\right) 
+c_0\times C'(m), 
\end{align*}
where  $C'(m)$ is an upper bound for the second integral, which is finite because function $g$ is bounded on the interval $[\alpha_1, 1-\alpha_1]$ and since random variable $T_2-T_1$ has a finite expectation (and even some exponential moment) since the diffusion coefficient is non-degenerate on the same interval  $[\alpha_1, 1-\alpha_1]$. Corollary \ref{Cor2} is proved. 
\end{proof}

The final result of this section states a bound for the convergence rate to the invariant measure, which is exponential in time, although, non-uniform in the initial state. Simultaneously, it shows the uniqueness of the invariant measure. 

\begin{Th}\label{T2}
Under the assumptions of theorel \ref{T1} there exist constants $c>0$, $0<\alpha<\frac{1}{2}$, $m>0$ such that for any 
$t\ge 0$ the inequality holds, 
\begin{equation}\label{T2-1}
||\mu_t^x-\mu||_{TV}\leq 2\, \left\{\left(C(m)c\alpha^{m+1}((1-x)^{-m}+x^{-m})+2\right)e^{-ct}\wedge 1\right\},
\end{equation}
with
$$
C(m)=\max\left(\frac{2}{mb_0-\frac{\epsilon^2}{2}m(m+1)\mu},\frac{2}{-mb_1-\frac{\epsilon^2}{2}m(m+1)\mu}\right).
$$
In particuar, the invariant measure $\mu$ is unique. 
\end{Th}

\begin{proof} 
Let us choose $m$ as proposed in (\ref{m2}) and consider two processes $(X,W)$ and $(Y,B)$, both  {\bf independent} solutions of the same SDE (\ref{eq1}) with initial conditions $X_0 = x$ and ${\cal L}(Y_0) = \mu$ respectively, where $\mu$ is the invariant measure of our strong Markov process on $(0,1)$ which existence is guaranteed by corollary \ref{Cor1}, and ${\cal L}(Y_0)$ stads for the distribution of the random variable $Y_0$. {\it (NB: At this point we do not know yet whether this measure is unique, but any invariant measure satisfying the bound (\ref{muint}) suffices for what follows.) } Now let us extend both processes on the direct product of the corresponding probability spaces; both couples $(X,W)$ and $(Y,B)$ will remain independent there. Now let
$$
M_1:= \inf(t\ge 0: \vert X_t - \frac12\vert \wedge  \vert Y_t - \frac12\vert \le \frac12 - \alpha_2), \quad  
L:= \inf(t\ge 0: \, X_t = Y_t).
$$
It follows from theorem \ref{T1} that the stopping time $M_1$ is almost surely finite and possesses some finite exponential moment similar to (\ref{exptau3}). Indeed, it follows similarly\footnote{It might be made a separate lemma as in \cite{6}, or in \cite{AYV1}, but it really looks  straightforward.} from the choice of a Lyapunov function as $\exp(ct)(x^{-m} + y^{-m})$, possibly after an appropriate decrease of the value of $\alpha_2$ (\ref{a12}).

\medskip

Further, let $\alpha_0 \in (0,\alpha_1)$, and let
$$
\hat M_1:= \inf(t\ge M_1: \, X_t \wedge Y_t\le \alpha_0, \; \text{or} \; X_t \vee Y_t \ge 1-\alpha_0) \wedge (M_1+1).
$$

Consider the interval of time $[M_1, \hat M_1]$. If it happens that $X_{M_1} = Y_{M_1}$ then we can conclude that $L\le M_1$. If $X_{M_1} \neq Y_{M_1}$ then either $X_{M_1} <Y_{M_1}$, or $X_{M_1}>Y_{M_1}$. In any of these two cases the probability that at time $M_1+1$ the inequality changed to the opposite one and both trajectories have not heat the boundary of the interval $(\alpha_0, 1-\alpha_0)$ is positive and bounded away from zero. Therefore, $L\le \hat M_1$ also with a positive probability bounded away from zero. 

If the event $L\le \hat M_1$ did not realise on the time interval  $M_1,  \hat M_1$, we will repeat the same procedure defining the moments by induction, 
$$
M_{n+1}:= \inf(t\ge \hat M_n : \vert X_t - \frac12\vert \wedge  \vert Y_t - \frac12\vert \le \frac12 - \alpha_2),
$$
$$
\hat M_{n+1}:= \inf(t\ge M_n: \, X_t \wedge Y_t \le \alpha_0, \; \text{or} \; X_t \vee Y_t\ge 1-\alpha_0) \wedge (M_n+1), \quad n\ge 1.
$$

On each interval $[M_n, \hat M_n]$ the probability of the event $L\le \hat M_n$ is positive and bounded away from zero by the same positive constant. 
On the other hand, by virtue of theorem $\ref{T1}$ the difference $M_{n+1} - \hat M_n$ 
has some finite exponential moment 
$E\exp(c_1(M_{n+1} - \hat M_n))\le c_2 < \infty$, 
and, more than that, 
\begin{equation}\label{nge1}
\sup_{n\ge 1} E\exp(c_1(M_{n+1} - \hat M_n))\le c_2 < \infty,
\end{equation}
both with the same values $c_1,c_2>0$.
The reason for this uniform bound is that 
$$
\max(|X_{\hat M_n} -\frac12|,|Y_{\hat M_n} -\frac12|) \le \frac12 - \alpha_0<\frac12.
$$

Note that $n=0$ is not included in the left hand side of (\ref{nge1}) because for $n=0$ we have a non-uniform bound 
\begin{equation}\label{n0}
E_{x,\mu}\exp(c_1(M_{1} - \hat M_0))\le 
C\alpha_2^{m+1}((1-x)^{-m}+x^{-m})+1
\end{equation}
due to theorem \ref{T1} and corollary \ref{Cor2}. The details of the hint leading to this bound are omitted here: they can be found in \cite{6}, as well as the little calculus which justifies the following inequality on the basis of the bounds (\ref{nge1}) and (\ref{n0}): 

\begin{equation}\label{Lexp}
P(L>t) \le D(x)\exp(-\lambda t)
\end{equation}
with some positive $\lambda$ and  
$$
D(x) = C\alpha_2^{m+1}((1-x)^{-m}+x^{-m})+1.
$$
The final bound on the distance in total variation $\|\mu_t^x-\mu\|_{TV}$ follows from the analogue of the ``coupling inequality'' (we call it an analogue because the idea of intersections was used here instead of the  coupling method)
\begin{align*}
\|\mu_t^x-\mu\|_{TV} \le 2 P(L>t). 
\end{align*}
In turn, the latter bound follows from the strong Markov property of the process starting from any initial distribution and from the replacement of the process $X_t$ by its equivalent in the distribution sense and coinciding with $Y_t$ for all $t\ge L$, as in the method of coupling; see, for example, \cite{6}; this new process may be denoted by $\tilde X_t$ although it will not be used in this paper,
$$
\tilde X_t := X_t 1(t < L) + Y_t 1(t \ge L). 
$$
Finally, the uniqueness of the invariant measure is the consequence of the bound (\ref{T2-1}). Theorem \ref{T2} is proved. 
\end{proof}

\section*{Acknowledgements}
%For the second author who proved theorem \ref{T2} and corollary \ref{Cor2} this work is funded by the Foundation for the Advancement of Theoretical Physics and Mathematics 
% in English for the second one!
%``BASIS'';
For the second author who proved theorem \ref{T2} and corollary \ref{Cor2} 
this study was funded by the Russian Foundation for Basic Research grant 20-01-00575a. 
All remaining results belong to the first author. 
% in English - to the first one. 

\end{document}